\documentclass[reqno,hidelinks,a4paper]{amsart}
\usepackage[T1]{fontenc}
\usepackage{latexsym, amssymb, amsmath}
\usepackage{enumitem}
\usepackage{amsthm}
\usepackage{tikz-cd}
\usepackage{orcidlink}
\usepackage{doi}
\renewcommand{\doi}[1]{DOI:~\href{https://doi.org/#1}{#1}}
\usepackage{url}
\usepackage[utf8]{inputenc}
\usepackage[british]{babel}
\usepackage[all]{xy}
\usepackage{hyperref}
\usepackage{calrsfs}
\usepackage{mathabx}
\usepackage{xcolor}
\usepackage{blkarray}
\usepackage{microtype}
\usepackage[textsize=scriptsize]{todonotes}
\usepackage{thmtools,mathtools}
\usepackage{nameref,hyperref,cleveref}
\usepackage{comment}

\theoremstyle{plain}
\newtheorem{theorem}[subsection]{Theorem}
\newtheorem{proposition}[subsection]{Proposition}
	
\newtheorem{lemma}[subsection]{Lemma}

\theoremstyle{definition}
\newtheorem{definition}[subsection]{Definition}

\theoremstyle{remark}
\newtheorem{remark}[subsection]{Remark}
\newtheorem{example}[subsection]{Example}

\numberwithin{equation}{section}

\newcommand{\noproof}{\hfil\qed}

\newcommand*{\CondName}[1]{{\rm (#1)}}
\newcommand*{\LACC}{\CondName{LACC}}

\newcommand{\C}{\ensuremath{\mathcal{C}}}
\newcommand{\F}{\mathbb{F}}         
\newcommand{\E}{\mathcal{E}}
\newcommand{\Z}{\mathbb{Z}}

\DeclareMathOperator{\bchar}{char}
\DeclareMathOperator{\bZ}{Z}
\DeclareMathOperator{\Ker}{Ker}
\DeclareMathOperator{\Aut}{Aut}
\DeclareMathOperator{\Act}{Act}
\DeclareMathOperator{\DAct}{DAct}
\DeclareMathOperator{\Hom}{Hom}
\DeclareMathOperator{\SplExt}{SplExt}
\DeclareMathOperator{\Imm}{Im}
\DeclareMathOperator{\Der}{Der}
\DeclareMathOperator{\Inn}{Inn}

\DeclareMathOperator{\USGA}{USGA}
\DeclareMathOperator{\id}{id}

\newcommand{\Grp}{\ensuremath{\mathsf{Grp}}}
\newcommand{\Ab}{\ensuremath{\mathsf{Ab}}}
\newcommand{\Set}{\ensuremath{\mathsf{Set}}}
\newcommand{\Assoc}{\ensuremath{\mathsf{Assoc}}}
\newcommand{\Lie}{\ensuremath{\mathsf{Lie}}}
\newcommand{\Leib}{\ensuremath{\mathsf{Leib}}}

\newcommand{\AbAlg}{\ensuremath{\mathsf{AbAlg}}}
\newcommand{\Nil}{\ensuremath{\mathsf{Nil}}}
\newcommand{\SplExtX}{\ensuremath{\mathsf{SplExt}(X)}}

\newdir{>}{{}*:(1,-.2)@^{>}*:(1,+.2)@_{>}}
\newdir{<}{{}*:(1,+.2)@^{<}*:(1,-.2)@_{<}}
\newdir{>>}{{}*!/3.5pt/:(1,-.2)@^{>}*!/3.5pt/:(1,+.2)@_{>}*!/7pt/:(1,-.2)@^{>}*!/7pt/:(1,+.2)@_{>}}
\newdir{ >}{{}*!/-8pt/@{>}}

\date{}

\begin{document}

\title[Weak action representability of \texorpdfstring{\(2\)}{2}-nilpotent groups]{Weak action representability\\of \texorpdfstring{\(2\)}{2}-nilpotent groups}

\author[A.~Dioguardi Burgio]{Alessandro Dioguardi Burgio~\orcidlink{0000-0003-0856-0144}}
\author[M.~Mancini]{Manuel Mancini~\orcidlink{0000-0003-2142-6193}}
\author[T.~Van~der Linden]{Tim Van~der Linden~\orcidlink{0000-0001-5474-5356}}

\email{alessandro.dioguardiburgio@unipa.it}
\email{manuel.mancini@uclouvain.be; manuel.mancini@unipa.it}
\email{tim.vanderlinden@uclouvain.be;  tim.van.der.linden@vub.be}

\address[Alessandro Dioguardi Burgio, Manuel Mancini]{Dipartimento di Matematica e Informatica, Università degli Studi di Palermo, Via Archirafi 34, 90123 Palermo, Italy.}

\address[Manuel Mancini, Tim Van der Linden]{Institut de Recherche en Mathématique et Physique, Université catholique de Louvain, chemin du cyclotron 2 bte L7.01.02, B--1348 Louvain-la-Neuve, Belgium.}

\address[Tim Van der Linden]{Mathematics \& Data Science, Vrije Universiteit Brussel, Pleinlaan 2, B--1050 Brussels, Belgium}

\begin{abstract}
    In this article, we investigate the representability of actions of the category \(\Nil_2(\Grp)\) of \(2\)-nilpotent groups. We first provide an algebraic characterisation of derived actions in \(\Nil_2(\Grp)\) in terms of the group \(\Aut_c(X)\) of central automorphisms of an object \(X\). We also characterise the morphisms \(B \to \Aut_c(X)\) that define an action of \(B\) on \(X\) in \(\Nil_2(\Grp)\).

    We then show that \(\Nil_2(\Grp)\) is not action representable, and that the existence of a weak representation is related to the amalgamation property. We prove that the category \(\Nil_2(\Grp)\) is weakly action representable, by exhibiting a strong amalgam of a suitable family of abelian subgroups of \(\Aut_c(X)\): an explicit weak representing object of \(X\) is the abelian group \(\Hom(X/[X,X],\bZ(X))\).

    Finally, we show that \(\Nil_2(\Grp)\) is not locally algebraically cartesian closed.
\end{abstract}


\maketitle

\section*{Introduction}\label{Introduction}
\emph{Internal actions} were defined in~\cite{BJK} by F.~Borceux, G.~Janelidze and G.~M.~Kelly in order to extend the correspondence between actions and split extensions from the category of groups to arbitrary semi-abelian categories~\cite{Semi-Ab}. In some of those categories, internal actions are exceptionally well behaved, in the sense that the actions on each object~\(X\) are \emph{representable}. More explicitly, this means that there exists an object~\([X]\) such that the functor
\[
    \Act(-,X)\cong \SplExt(-,X),
\]
sending an object~\(B\) to the set of actions of~\(B\) on~\(X\) (equivalently, to the set of isomorphism classes of split extensions of~\(B\) by~\(X\)), is naturally isomorphic to the functor \(\Hom(-,[X])\).

The notion of action representability in a semi-abelian category was further investigated in~\cite{BJK2}, where it is shown there that the category of commutative associative algebras over a field is not action representable. Subsequently, it was proved in~\cite{Tim} that, over an infinite field of characteristic different from~\(2\), the only action representable varieties of non-associative algebras are the category \(\Lie\) of Lie algebras and the category \(\AbAlg\) of abelian algebras. The strength of this notion naturally motivated the introduction of weaker related concepts.

In~\cite{act_accessible}, D.~Bourn and G.~Janelidze introduced the notion of an \emph{action accessible} category so as to encompass relevant examples that do not fit into the frame of action representable categories, such as rings, associative algebras, and Leibniz algebras~\cite{loday}. It was later proved by A.~Montoli in~\cite{Montoli} that every \emph{Orzech category of interest}~\cite{Orzech} is action accessible. Moreover, in~\cite{Casas}, the authors showed that every Orzech category of interest admits a weaker form of representing object, called the \emph{universal strict general actor} (USGA).

More recently, G.~Janelidze introduced in~\cite{WRA} the notion of a \emph{weakly action representable} category. Rather than requiring that, for every object \(X\) of a semi-abelian category~\(\C\), there exist an object \([X]\) and a natural isomorphism
\[
    \Act(-,X)\cong\Hom_{\C}(-,[X]),
\]
one only asks for the existence of an object \(T=T(X)\) of \(\C\) together with a natural monomorphism of functors
\[
    \Act(-,X)\rightarrowtail \Hom_{\C}(-,T).
\]
Such an object \(T\) is called a \emph{weak representing object} of \(X\). Examples of weakly action representable categories include the category of associative algebras~\cite{WRA} and the category of Leibniz algebras~\cite{CigoliManciniMetere}.

The notion of weak action representability was further investigated in~\cite{WRAAlg} in the setting of \emph{operadic} varieties of non-associative algebras over a field. In that work, the authors exhibited new examples of weakly action representable categories, including the varieties of \(2\)-nilpotent (commutative and anti-commutative) algebras.

Another important observation made by  G.~Janelidze in~\cite{WRA} is that every weakly action representable category is action accessible. The converse, however, does not hold. Indeed, J.~R.~A.~Gray observed in~\cite[Proposition~2.2]{Gray} that weak action representability is closely related to the existence of a so-called \emph{amalgam}~\cite{kiss}, a notion that had already appeared in~\cite{BJK2} in connection with action representability, and he proved that the varieties of \(n\)-solvable groups are not weakly action representable, for any \(n \geq 3\). This result was later extended in~\cite{XabiMancini} to the varieties of \(2\)-solvable groups and of \(k\)-nilpotent groups, for any \(k \geq 3\). There, the authors also proved that the varieties of \(n\)-solvable Lie algebras (\(n \geq 2\)) and \(k\)-nilpotent Lie algebras (\(k \geq 3\)) fail to be weakly action representable. These are the first examples of action accessible varieties of non-associative algebras which are not weakly action representable.

The main goal of this article is to investigate the representability of actions in the category \(\Nil_2(\Grp)\) of \(2\)-nilpotent groups. After recalling the necessary background in \Cref{SecPrel}, we give an algebraic characterisation of (external) actions in \(\Nil_2(\Grp)\). More precisely, we characterise the actions of a \(2\)-nilpotent group \(B\) on another \(2\)-nilpotent group~\(X\) such that the corresponding semidirect product \(B \ltimes X\) is still \(2\)-nilpotent.

As a consequence, we obtain a natural embedding of the actions on~\(X\) into the group \(\Aut_c(X)\) of \emph{central automorphisms} of~\(X\); see \cite{Adney-Yen}. We also provide a complete characterisation of the \emph{acting morphism}, that is, of those group homomorphisms \(B \to \Aut_c(X)\) that define an action of \(B\) on \(X\) in the category \(\Nil_2(\Grp)\).

In \Cref{sec_rep}, we investigate weak action representability in the category \(\Nil_2(\Grp)\). We first prove in \Cref{thm_ar} that \(\Nil_2(\Grp)\) is not action representable. Then, we show that, for a \(2\)-nilpotent group \(X\), the existence of a weak representation for the actions on \(X\) is closely related to the so-called \emph{amalgamation property} (AP)~\cite{kiss}. Finally, we prove that \(\Nil_2(\Grp)\) is weakly action representable, by exhibiting a strong amalgam of a suitable family of abelian subgroups of \(\Aut_c(X)\): a weak representing object of \(X\) is the abelian group \(\Hom(X/[X,X],\bZ(X))\) of group homomorphisms from the abelianisation of \(X\) to its centre.

In \Cref{sec_LACC} we study \emph{local algebraic cartesian closedness} \LACC{} for the category \(\Nil_2(\Grp)\), a notion introduced by J.~R.~A.~Gray in his Ph.D.\ thesis~\cite{GrayThesis} in connection with the concept of \emph{algebraic exponentiation} (see~\cite{BournGray,GrayACS}). In the semi-abelian setting, this property provides a natural categorical framework for the study of commutators and internal actions. In the context of varieties of non-associative algebras over a field, it has been shown that the condition \LACC{} essentially characterises the variety of Lie algebras (see~\cite{GM-VdL2,GM-VdL3,GrayLie}), so that, in this context, it is equivalent to action representability. At present, it remains an open problem whether local algebraic cartesian closedness implies (weak) action representability, or whether a counterexample exists. Here, using the characterisation of the condition \LACC{} asking that the comparison map
\[
    (B \flat X) + (B \flat Y) \to B \flat (X+Y),
\]
where \(B \flat X\) denotes the kernel of the morphism \([\id_B,0]\colon {B+X\to B}\), is an isomorphism, we prove that the category \(\Nil_2(\Grp)\) is not \LACC{}. This suggests a further line of investigation into the \LACC{} property for other subvarieties of \(\Grp\).

\section{Preliminaries}\label{SecPrel}

The notion of \emph{semi-abelian} category was introduced in~\cite{Semi-Ab} by G.~Janelidze, L.~Márki and W.~Tholen in order to provide a categorical setting which would capture categorical-algebraic properties of non-abelian algebraic structures, such as groups, rings and Lie algebras. We recall that a category~\(\C\) is
semi-abelian when it is pointed, Barr-exact~\cite{Barr}, Bourn-protomodular~\cite{Bourn} and has finite coproducts.

A notion that can be explored in the context of semi-abelian categories is that of split extension.

\begin{definition}
    Let \(\C\) be a semi-abelian category and let \(B\), \(X\) be objects of \(\C\). A \emph{split extension} of \(B\) by \(X\) is a diagram
    \begin{equation*}\label{eq:split_ext}
        \begin{tikzcd}
            X\arrow [r, "k"]
            &A \arrow[r, shift left, "p"] &
            B\ar[l, shift left, "s"]
        \end{tikzcd}
    \end{equation*}
    in \(\C\) such that \(p \circ s = \id_B\) and \((X,k)\) is a kernel of \(p\).
\end{definition}

For any object \(X\) of \(\C\), one can define the functor
\[
    \SplExt(-,X) \colon \C^{op} \to \Set
\]
which maps any object \(B\) of \(\C\) to the set \(\SplExt(B,X)\) of isomorphism classes of split extensions of \(B\) by \(X\), and any arrow \(f\colon B'\to B\) to the \emph{change of base} morphism \(f^*\colon \SplExt(B,X) \to \SplExt(B',X)\) given by pulling back along \(f\).

A feature of semi-abelian categories is that one can define a notion of \emph{internal action}. Internal actions correspond to split extensions via a semidirect product construction; it turns out that for our purposes we need no explicit description of what is an internal action, since split extensions are easier to work with, especially in categories of groups with operations. We refer the interested reader to~\cite{BJK2}, where the equivalence between the two concepts is described in detail.

For us here, it is sufficient to observe that if we fix an
object \(X\) of \(\C\), internal actions on \(X\) give rise to a functor
\[
    \Act(-,X)\colon \C^{op} \to \Set
\]
and a natural isomorphism \(\Act(-,X)\cong \SplExt(-,X)\) (see~\cite{BJK2}). This justifies the terminology in the definition that follows.

\begin{definition}\cite{BJK2}
    A semi-abelian category \(\C\) is said to be \emph{action representable} if for any object \(X\) in \(\C\), the functor \(\SplExt(-,X)\) is representable. This means that there exists an object \([X]\) of \(\C\) and a natural isomorphism of functors
    \[
        \SplExt(-,X) \cong \Hom_{\C}(-,[X]).
    \]
\end{definition}

The prototype examples of action representable categories are the category \(\Grp\) of groups and the category \(\Lie\) of Lie algebras over a commutative unital ring. In the first case, every action of \(B\) on \(X\) is represented by a group homomorphism \(B \to \Aut(X)\), where \(\Aut(X)\) is the group of automorphisms of \(X\). In the second case, any action of \(B\) on \(X\) is represented by a Lie algebra homomorphism \(B \to \Der(X)\), where \(\Der(X)\) is the Lie algebra of derivations of \(X\).

However, the notion of an action representable category has proved to be quite restrictive: for instance, the authors of~\cite{Tim} showed that, over an infinite field \(\F\) with \(\bchar(\F)\neq 2\), the only action representable varieties of non-associative algebras~\cite{VdL-NAA} are the variety of abelian algebras and the variety of Lie algebras.

\begin{remark}\label{BJ}
    We recall that a semi-abelian category \(\C\) is action representable if and only if, for every object \(X\) of \(\C\), the category \(\SplExtX\) of split extensions in~\(\C\) with kernel \(X\) admits a terminal object of the form
    \[
        \begin{tikzcd}
            X \arrow[r]
            & {[X] \ltimes X} \arrow[r, shift left] &
            {[X]\text{.}} \arrow[l, shift left]
        \end{tikzcd}
    \]
    This condition can be relaxed by requiring only that every object of \(\SplExtX\) is \emph{accessible}, i.e., it admits a morphism into a \emph{subterminal} (or \emph{faithful}) object. Here, we recall that an object is faithful if there exists at most one morphism into it. This gives rise to the following notion, introduced by D.~Bourn and G.~Janelidze in~\cite{act_accessible} in order to compute centralisers of normal subobjects and equivalence relations.
\end{remark}

\begin{definition}\cite{act_accessible}
    A semi-abelian category \(\C\) is \emph{action accessible} if, for any object \(X\) of \(\C\), every object in \(\SplExtX\) is accessible.
\end{definition}

The concept of action accessibility makes it possible to include a broader class of categories that fail to be action representable, such as the category of rings and the variety of associative algebras.

We observe that, since the existence of a terminal object in \(\SplExtX\) is stronger than every object being accessible, it follows from \Cref{BJ} that
\[
    \emph{action representability}\Rightarrow\emph{action accessibility.}
\]

A different way to weaken action representability is to relax the requirements on the representing object~\([X]\), rather than the existence of a terminal object in \(\SplExtX\). This is the approach adopted in~\cite{Casas}, where it was shown that every \emph{Orzech category of interest}~\cite{Orzech} admits a so-called \emph{universal strict general actor} (USGA). It was then proved by A.~Montoli in~\cite{Montoli} that every Orzech category of interest is action accessible. Moreover, it was shown in~\cite{CigoliManciniMetere} that, for every object~\(X\) of an Orzech category of interest \(\C\), there exists a natural monomorphism of functors
\[
    \tau \colon \SplExt(-,X) \rightarrowtail \Hom_{\C'}(U(-),\USGA(X)),
\]
where \(\C'\) is a category containing \(\C\) as a full subcategory, and \(U \colon \C \to \C'\) is the forgetful functor.

More recently, G.~Janelidze introduced in~\cite{WRA} the notion of weak action representability.

\begin{definition}\cite{WRA}
    A semi-abelian category \(\C\) is said to be \emph{weakly action representable} if for every object \(X\) in \(\C\), the functor \(\SplExt(-,X)\) admits a weak representation. This means that there exists an object \(T=T(X)\) of \(\C\) and a natural monomorphism of functors
    \[
        \tau\colon \SplExt(-,X) \rightarrowtail \Hom_{\C}(-,T).
    \]
    A morphism \((\varphi\colon B \to T) \in \Imm(\tau_B)\) is called \emph{acting morphism}.
\end{definition}

Notice that every action representable category is weakly action representable. In this case, \(T=[X]\) is the actor of \(X\), \(\tau\) is a natural isomorphism and every arrow \(\varphi\colon B \to [X]\) is an acting morphism. Other examples of weakly action representable categories are the variety \(\Assoc\) of associative algebras~\cite{WRA}, the variety \(\Leib\) of Leibniz algebras~\cite{CigoliManciniMetere}, and the varieties of \(2\)-nilpotent (commutative and anti-commutative) algebras (see~\cite{WRAAlg}). More generally, if in an Orzech category of interest \(\C\) each \(\USGA(X)\) is an object of the category, then \(\C\) is weakly action
representable (see~\cite[Corollary 4.2]{CigoliManciniMetere}).

Another important observation made by G.~Janelidze in~\cite{WRA} is that every weakly action representable category is action accessible. We thus have that
\[
    \textit{action representability}\Rightarrow\textit{weak action representability}\Rightarrow\textit{action accessibility.}
\]
The author ended with an open question: whether reasonably mild conditions may be found on a semi-abelian category under which the second implication may be reversed.

This question was addressed by J.~R.~A.\ Gray, who showed in~\cite{Gray} that the notion of weak action representability is related to the so-called \emph{amalgamation property} (AP)~\cite{kiss}. Using (AP), he established sufficient conditions under which a Birkhoff subcategory of an action representable category fails to be weakly action representable. In particular, he proved that the varieties of \(n\)-solvable groups are not weakly action representable for any \(n \geq 3\). This result was later adapted in~\cite{XabiMancini} to prove that the varieties of \(k\)-nilpotent groups (\(k \geq 3\)), as well as the variety of \(2\)-solvable groups, are not weakly action representable. In the same paper, it was also shown that the varieties of \(k\)-nilpotent Lie algebras (\(k \geq 3\)) and the varieties of \(n\)-solvable Lie algebras (\(n \geq 2\)) fail to be weakly action representable. Note, however, that this is no longer true for \(k=2\): indeed, it was proved in~\cite[Theorem~2.21]{WRAAlg} that the variety of \(2\)-nilpotent anti-commutative algebras over a field \(\F\) with \(\bchar(\F)\neq 2\) is weakly action representable, with a weak representing object on an object \(X\) being an abelian algebra.

In this article, we show that an analogue of Theorem~2.21 of~\cite{WRAAlg} holds for the variety \(\Nil_2(\Grp)\) of \(2\)-nilpotent groups. More precisely, in the next section we provide an algebraic characterisation of the actions on a \(2\)-nilpotent group \(X\) in the variety of \(2\)-nilpotent groups, in terms of central automorphisms of~\(X\). We then use this description in \Cref{sec_rep} to prove that \(\Nil_2(\Grp)\) is weakly action representable, and that a weak representing object of \(X\) is an abelian group.

\section{Actions and central automorphisms of \texorpdfstring{\(2\)}{2}-nilpotent groups}\label{sec_act}

In this section we aim to describe actions in the variety of \(2\)-nilpotent groups. We start by recalling the following definition.

\begin{definition}
    Let \(X\) be a group. The \emph{lower central series} of \(X\) is defined as
    \[
        X^{0}=X, \quad X^{k}=[X^{k-1},X],
    \]
    for any \(k \geq 1\).

    The group \(X\) is said to be \emph{nilpotent of class \(k\)} if \(X^{k-1}\neq \{1\}\) and \(X^{k}=\{1\}\).
\end{definition}

We denote by \(\Nil_k(\Grp)\) the subvariety of \(\Grp\) consisting of all groups of nilpotency class at most~\(k\).

\begin{remark}
    Note that the class of all nilpotent groups (i.e., for any nilpotency class) does not form a variety; however, it does when the nilpotency class is bounded. For this reason, in the rest of the paper, when we say that a group \(X\) is \(2\)-nilpotent, we mean that \(X\) is an object of \(\Nil_2(\Grp)\), i.e., \(X\) is nilpotent of class at most \(2\). More precisely, the category \(\Nil_2(\Grp)\) consists of all groups \(X\) such that the \emph{commutator subgroup}
    \[
        [X,X]=\langle [x,y] \mid x,y \in X \rangle
    \]
    lies in the center \(\bZ(X)\) of \(X\). Here, as usual,
    \[
        [x,y]=xyx^{-1}y^{-1}
    \]
    denotes the commutator of the elements \(x\) and \(y\).
\end{remark}

Since the variety \(\Nil_2(\Grp)\) is an Orzech category of interest~\cite{Orzech}, it is convenient to describe internal actions in terms of the so-called \emph{derived actions}.

\begin{definition}
    Let
    \begin{equation}\label{diag:Spl}
        \begin{tikzcd}
            X\arrow [r, "k"]
            &A \arrow[r, shift left, "p"] &
            B\ar[l, shift left, "s"]
        \end{tikzcd}
    \end{equation}
    be a split extension of \(2\)-nilpotent groups. The map
    \[
        \xi \colon B \times X \to X \colon (b,x) \mapsto b \ast x
    \]
    defined by
    \[
        b \ast x = s(b)k(x)s(b)^{-1},
    \]
    is called the \emph{derived action} of \(B\) on \(X\) associated with \eqref{diag:Spl}.
\end{definition}

In fact, for every Orzech category of interest there is a natural isomorphism of functors
\[
    \Act(-,X) \cong \DAct(-,X),
\]
where \(\DAct(-,X)\) denotes the functor of derived actions on a fixed object \(X\). We refer the reader to~\cite[Remark 3.16 and Proposition 3.17]{Mancini_Thesis} for a detailed proof.

Given a map
\[
    \xi \colon B \times X \to X \colon (b,x) \mapsto b \ast x,
\]
one can always define a binary operation on the cartesian product \(B \times X\) by
\[
    (b,x) \cdot (b',x')=(bb',x (b \ast x')).
\]
Via~\cite[Theorem 2.4]{Orzech}, this defines a \(2\)-nilpotent group structure on \(B \times X\) if and only if \(\xi\) is a derived action of \(B\) on \(X\) in \(\Nil_2(\Grp)\). This, in turn, is equivalent to a set of conditions on the map \(\xi\), as explained in the following proposition, which is a special case of in~\cite[Proposition 1.1]{Datuashvili}.

\begin{proposition}\label{PropAct}
    \((B \times X,\cdot)\) is a \(2\)-nilpotent group if and only if the following conditions hold:
    \begin{enumerate}
        \item \(\xi\) is a derived action of \(B\) on \(X\) in the category \(\Grp\) of groups, that is:
              \begin{itemize}
                  \item[(i)] \(1 \ast x = x\);
                  \item[(ii)] \((bb') \ast x = b \ast (b' \ast x)\);
                  \item[(iii)] \(b \ast (xx') = (b \ast x)(b \ast x')\);
              \end{itemize}

        \item \((b \ast x)x^{-1} \in \bZ(X)\cap X^B\),
    \end{enumerate}
    for any \(b\), \(b' \in B\) and \(x\), \(x' \in X\), where
    \[
        X^B=\{x \in X \mid b\ast x=x, \; \forall b\in B\}
    \]
    is the set of fixed points of the action\footnote{If \(\xi\) is a derived action of \(B\) on \(X\) in \(\Grp\), then one may check that \(X^B\) is a subgroup of \(X\).}. The resulting \(2\)-nilpotent group is the \emph{semidirect product} of \(B\) and \(X\), denoted by \(B \ltimes X\).
\end{proposition}

\begin{proof}
    Suppose that \((B \times X,\cdot)\) is a \(2\)-nilpotent group. Then, \(\xi\) is a derived action in the category \(\Grp\). Moreover, for any \((b,x)\), \((b',x')\), \((c,y)\in B\times X\), the equality
    \[
        [(b,x),(b',x')] \cdot (c,y)=(c,y) \cdot [(b,x),(b',x')]
    \]
    gives
    \begin{equation}\label{eq_2nil}
        \begin{split}
            x(b \ast x')([b,b']b' \ast x^{-1}) & ([b,b'] \ast (x'^{-1}y))                                                                            \\
            =                                  & y\, c \ast \bigl(x(b \ast x')\bigl([b,b']b' \ast x^{-1}\bigr)\bigl([b,b'] \ast x'^{-1}\bigr)\bigr).
        \end{split}
    \end{equation}
    Now, for \(b'=c=1\) and \(x=1\), using the fact that \(b \ast 1=1\), we obtain
    \[
        (b \ast x')x'^{-1}y=y(b \ast x')x'^{-1},
    \]
    and therefore
    \[
        (b \ast x')x'^{-1}\in \bZ(X).
    \]
    Similarly, if we take \(b'=1\) and \(x=y=1\), we get
    \[
        c \ast \bigl((b \ast x')x'^{-1}\bigr)=(b \ast x')x'^{-1},
    \]
    so that
    \[
        (b \ast x')x'^{-1}\in X^B.
    \]
    Conversely, suppose that conditions~(1) and~(2) hold. Since \(\xi\) is a derived action of~\(B\) on \(X\) in \(\Grp\), it follows that \((B \times X,\cdot)\) is a group. It remains to prove that it is \(2\)-nilpotent, that is,
    \[
        [(b,x),(b',x')] \cdot (c,y)=(c,y)\cdot [(b,x),(b',x')]
    \]
    for every \((b,x)\), \((b',x')\), \((c,y)\in B\times X\). Since both \(B\) and \(X\) are \(2\)-nilpotent, it is enough to verify that \eqref{eq_2nil} holds.

    Now, the condition \((b \ast x)x^{-1}\in X^B\) implies that
    \[
        (bb')\ast x=(b'b)\ast x
    \]
    for any \(b\), \(b'\in B\) and \(x\in X\). Indeed, one has
    \begin{align*}
        ((bb') \ast x)((b'b) \ast x)^{-1} = & ((bb') \ast x)((b'b) \ast x^{-1})                                                      \\
        =                                   & (b \ast (b' \ast x)) (b' \ast (b \ast x^{-1}))                                         \\
        =                                   & (b \ast ((b' \ast x) x^{-1})) (b \ast x) (b' \ast ((b \ast x^{-1})x))(b' \ast x^{-1} ) \\
        =                                   & (b' \ast x)x^{-1} (b \ast x) (b \ast x^{-1})x (b' \ast x^{-1} )                        \\
        =                                   & (b' \ast x)x^{-1} x (b' \ast x^{-1} )                                                  \\
        =                                   & (b' \ast x)(b' \ast x^{-1})=1.
    \end{align*}
    Therefore, \([b,b'] \ast x=x\) and \eqref{eq_2nil} reduces to
    \[
        gy = y(c\ast g),
    \]
    where
    \[
        g=x(b\ast x')(b'\ast x^{-1})x'^{-1}.
    \]
    To conclude, observe that
    \[
        g=\bigl((b'\ast x)x^{-1}\bigr)^{-1}[\,b'\ast x,\; b\ast x'\,](b\ast x')x'^{-1}.
    \]
    Moreover, the commutator subgroup \([X,X]\) is fixed by the action of \(B\). Indeed, let \(b\in B\) and \(x\), \(x'\in X\). By condition~(2), there exist \(z_x\), \(z_{x'}\in \bZ(X)\) such that
    \[
        b\ast x = z_x x
        \quad\text{and}\quad
        b\ast x' = z_{x'}x'.
    \]
    It follows that
    \[
        b\ast [x,x']=[b\ast x,b\ast x']=[z_xx,z_{x'}x']=[x,x'],
    \]
    since \(z_x\) and \(z_{x'}\) are central.

    Now, since \(X^B\) is a subgroup of \(X\) and
    \[
        [X,X]\leq \bZ(X)\cap X^B,
    \]
    we deduce from the above expression for \(g\) that
    \[
        g\in \bZ(X)\cap X^B.
    \]
    Hence
    \[
        gy = yg = y(c\ast g),
    \]
    so that \eqref{eq_2nil} holds. Therefore, \((B \times X,\cdot)\) is a \(2\)-nilpotent group.
\end{proof}

\begin{remark}
    We observe that, for any split extension \eqref{diag:Spl} and the corresponding derived action
    \[
        \xi \colon B \times X \to X,
    \]
    there is an isomorphism of split extensions
    \begin{equation*}
        \begin{tikzcd}
            X \arrow[r, "\iota_2"] \arrow[d, equal]
            & B \ltimes X \arrow[r, shift left, "\pi_1"] \arrow[d, "\theta"]
            & B \arrow[l, shift left, "\iota_1"] \arrow[d, equal] \\
            X \arrow[r, "k"]
            & A \arrow[r, shift left, "p"]
            & B \arrow[l, shift left, "s"]
        \end{tikzcd}
    \end{equation*}
    where \(\iota_1\), \(\iota_2\), and \(\pi_1\) are the canonical injections and projection, and \(\theta\) is defined by
    \[
        \theta(b,x)=k(x)s(b),
    \]
    for every \((b,x)\in B\ltimes X\).
\end{remark}

\begin{remark}\label{rmk}
    Since \(\xi\) is an action of \(B\) on \(X\) in the category \(\Grp\), the assignment
    \[
        b \mapsto \varphi_b = b \ast -
    \]
    defines a group homomorphism \(\varphi \colon B \to \Aut(X)\). Moreover, the condition
    \[
        (b \ast x)x^{-1} \in \bZ(X)
    \]
    means that \(\varphi_b\) is a \emph{central automorphism} of \(X\):
    \[
        \varphi_b \in \Aut_c(X)=\{ f \in \Aut(X) \mid f(x)x^{-1} \in \bZ(X), \; \forall x \in X \}.
    \]
    It is straightforward to check that \(\Aut_c(X)\) is a normal subgroup of \(\Aut(X)\), but is not, in general, an object of \(\Nil_2(\Grp)\). For instance, if \(X=\Z_2\times \Z_2\), then
    \[
        \Aut_c(X)=\Aut(X)=\operatorname{GL}_2(\Z_2) \cong S_3
    \]
    which is not nilpotent.

    Conversely, let
    \[
        \varphi \colon B \to \Aut_c(X) \colon b \mapsto b \ast -
    \]
    be a group homomorphism such that
    \[
        b' \ast \bigl((b \ast x)x^{-1}\bigr)=(b \ast x)x^{-1}
    \]
    for any \(b\), \(b' \in B\) and \(x \in X\). Then, one can define a split extension
    \begin{equation*}
        \begin{tikzcd}
            X \arrow[r, "\iota_2"]
            & B \ltimes X \arrow[r, shift left, "\pi_1"]
            & B \arrow[l, shift left, "\iota_1"]
        \end{tikzcd}
    \end{equation*}
    in \(\Nil_2(\Grp)\), where the group structure on \(B \ltimes X\) is defined by
    \[
        (b,x)\cdot (b',x')=(bb',x(b\ast x')).
    \]
    However, a generic morphism \(B \to \Aut_c(X)\) need not give rise to a split extension in \(\Nil_2(\Grp)\), as the following example shows.
\end{remark}

\begin{example}
    Consider the \(2\)-nilpotent group
    \[
        B = \langle x, a, b \mid [x, a]=b^{-1}, [x, b] = [a, b] = x^5 = a^5 = b^5=1 \rangle.
    \]
    Let \(X=\Z_5 \times \Z_5 \times \Z_5\) and let \(\varphi \colon B \rightarrowtail \Aut(X)\) be the group monomorphism defined by
    \[
        \varphi(x) =
        \begin{pmatrix}
            1 & 1 & 0 \\
            0 & 1 & 0 \\
            0 & 0 & 1
        \end{pmatrix}, \quad
        \varphi(a) =
        \begin{pmatrix}
            1 & 0 & 0 \\
            0 & 1 & 1 \\
            0 & 0 & 1
        \end{pmatrix}, \quad
        \varphi(b)=
        \begin{pmatrix}
            1 & 0 & -1 \\
            0 & 1 & 0  \\
            0 & 0 & 1
        \end{pmatrix}.
    \]
    It is immediate to see that the relations are preserved.

    Let \(\xi\) be the action of \(B\) on \(X\) induced by \(\varphi\), that is,
    \[
        g \ast \underline{v}=\varphi(g) \underline{v}, \quad \forall g \in B, \; \forall \underline{v}=\begin{pmatrix}
            v_1 \\ v_2 \\ v_3
        \end{pmatrix} \in X.
    \]
    Since \(X\) is abelian, every automorphism of \(X\) is central. However, \(\xi\) is not an action in the variety \(\Nil_2(\Grp)\), since
    \[
        a \ast \begin{pmatrix}
            0 \\ 1 \\ 1
        \end{pmatrix} + \begin{pmatrix}
            0 \\ -1 \\ -1
        \end{pmatrix} = \begin{pmatrix}
            1 & 0 & 0 \\
            0 & 1 & 1 \\
            0 & 0 & 1
        \end{pmatrix} \begin{pmatrix}
            0 \\ 1 \\ 1
        \end{pmatrix} + \begin{pmatrix}
            0 \\ -1 \\ -1
        \end{pmatrix} = \begin{pmatrix}
            0 \\ 1 \\ 0
        \end{pmatrix}
    \]
    and
    \[
        x* \Bigg( a \ast \begin{pmatrix}
                0 \\ 1 \\ 1
            \end{pmatrix} + \begin{pmatrix}
                0 \\ -1 \\ -1
            \end{pmatrix} \Bigg)=\begin{pmatrix}
            1 & 1 & 0 \\
            0 & 1 & 0 \\
            0 & 0 & 1
        \end{pmatrix}\begin{pmatrix}
            0 \\ 1 \\ 0
        \end{pmatrix}=\begin{pmatrix}
            1 \\ 1 \\ 0
        \end{pmatrix} \neq \begin{pmatrix}
            0 \\ 1 \\ 0
        \end{pmatrix}.
    \]
    In fact, it is shown in~\cite{XabiMancini} that the corresponding semidirect product \(B\ltimes X\) has nilpotency class \(3\).
\end{example}

We can now claim the following result.

\begin{theorem}\label{main_thm}
    Let \(X\) be a \(2\)-nilpotent group and let \(U \colon \Nil_2(\Grp) \to \Grp\) be the forgetful functor.
    \begin{itemize}
        \item[(i)] There exists a natural monomorphism of functors
              \[
                  \rho \colon \SplExt(-,X) \rightarrowtail \Hom_\Grp(U(-),\Aut_c(X)).
              \]

        \item[(ii)] A group homomorphism
              \[
                  \varphi \colon U(B) \to \Aut_c(X) \colon b \mapsto b \ast -
              \]
              belongs to \(\Imm(\rho_B)\) if and only if
              \begin{equation}\label{act_mor}
                  b' \ast ((b \ast x)x^{-1})=(b \ast x)x^{-1},
              \end{equation}
              for any \(b\), \(b' \in B\) and \(x \in X\). In this case, the image \(\varphi(B)\) is an abelian subgroup of \(\Aut_c(X)\).
    \end{itemize}
\end{theorem}

\begin{proof}{\ }
    \begin{itemize}
        \item[(i)] We define \(\rho\) as follows: for every \(2\)-nilpotent group \(B\), the component
              \[
                  \rho_B \colon \SplExt(B,X) \to \Hom_\Grp(U(B),\Aut_c(X))
              \]
              is the morphism in \(\Set\) which maps any action
              \[
                  \xi \colon B \times X \to X \colon (b,x) \mapsto b \ast x
              \]
              to the group homomorphism
              \[
                  \varphi \colon B \to \Aut_c(X) \colon b \mapsto b \ast -.
              \]
              The transformation \(\rho\) is natural. Indeed, let
              \[
                  f \colon B' \to B
              \]
              be a morphism in \(\Nil_2(\Grp)\), and let
              \[
                  f^*=\SplExt(f,X) \colon \SplExt(B,X) \to \SplExt(B',X)
              \]
              be the \emph{change of base} function given by pulling back along \(f\). Then, it is easy to verify that the diagram in \(\Set\)
              \[
                  \begin{tikzcd}
                      \SplExt(B,X) \arrow[r, "\rho_B"] \arrow[d, swap, "{f^\ast}"]
                      & \Hom_\Grp(U(B),\Aut_c(X)) \arrow[d,"{- \circ f}"]\\
                      \SplExt(B',X) \arrow[r, swap, "\rho_{B'}"] & \Hom_\Grp(U(B'),\Aut_c(X))
                  \end{tikzcd}
              \]
              is commutative. Moreover, the map \(\rho_B\) is an injection since every element of \(\SplExt(B,X)\) is uniquely determined by the corresponding derived action of \(B\) on \(X\), that is, by the map
              \[
                  \xi \colon B \times X \to X \colon (b,x) \mapsto b \ast x.
              \]
              Therefore, \(\rho\) is a natural monomorphism of functors.
        \item[(ii)] By \Cref{rmk}, a morphism \(\varphi \colon U(B) \to \Aut_c(X)\) defines an action of \(B\) on \(X\) in \(\Nil_2(\Grp)\) if and only if \eqref{act_mor} holds. Moreover, as observed in the proof of \Cref{PropAct}, the condition \eqref{act_mor} implies that
              \[
                  (bb') \ast x = (b'b) \ast x,
              \]
              for any \(b\), \(b' \in B\) and \(x \in X\). This means that the automorphisms \(b \ast -\) and \(b' \ast -\) commute, that is, \(\Imm(\varphi)\) is an abelian subgroup of \(\Aut_c(X)\).\qedhere
    \end{itemize}
\end{proof}

We conclude this section with an example of a canonical action of any \(2\)-nilpotent group \(X\) on itself given by inner automorphisms. Before doing so, we state the following lemma, whose proof is immediate.

\begin{lemma}\label{lemma}

    Let \(X\) be a group, and let \(\Inn(X)\) denote the group of \emph{inner automorphisms} of \(X\). Then
    \[
        \Inn(X) \leq \Aut_c(X)
    \]
    if and only if \(X\) is \(2\)-nilpotent. \noproof
\end{lemma}

\begin{example}
    Let \(X\) be a \(2\)-nilpotent group. Then, by \Cref{lemma}, \(\Inn(X)\) is a subgroup of \(\Aut_c(X)\). Now consider the canonical group homomorphism
    \[
        \Inn \colon X \to \Inn(X) \colon g \mapsto \varphi_g,
    \]
    where \(\varphi_g\)  is the inner automorphism defined by
    \[
        \varphi_g(x)=gxg^{-1},
    \]
    for any \(x \in X\). We have
    \[
        \varphi_h(\varphi_g(x)x^{-1})=\varphi_h([g,x])=[g,x]=\varphi_g(x)x^{-1},
    \]
    since \([g,x] \in \bZ(X)\). Therefore, by (ii) of \Cref{main_thm}, if \(i \colon \Inn(X) \hookrightarrow \Aut_c(X)\) denotes the inclusion, then
    \[
        \overline{\Inn} = i \circ \Inn \in \Imm(\rho_X).
    \]
    Finally, the action of \(X\) on itself induced by \(\overline{\Inn}\) is precisely the conjugation action:
    \[
        g \ast x = \varphi_g(x)=gxg^{-1},
    \]
    for any \(g\), \(x \in X\). It follows that the corresponding semidirect product \(X \ltimes X\) is \(2\)-nilpotent.
\end{example}

\section{Weak action representability of \texorpdfstring{\(2\)}{2}-nilpotent groups}\label{sec_rep}

The result of \Cref{main_thm} is closely related to~\cite[Corollary~4.2]{CigoliManciniMetere}. Indeed, by the construction in~\cite[Section~4]{Casas}, the universal strict general actor \(\USGA(X)\) of \(X\) in the Orzech category of interest \(\Nil_2(\Grp)\) is the subgroup of \(\Aut_c(X)\) generated by the images of the acting morphisms \(U(B)\to\Aut_c(X)\). In general, it is a proper subgroup: for instance, \(\USGA(\Z_3)\) is trivial, since the inversion automorphism of~\(\Z_3\) does not satisfy~\eqref{act_mor}, while \(\Aut_c(\Z_3)\cong\Z_2\). On the other hand, \(\USGA(X)\) is not, in general, \(2\)-nilpotent: for \(X=\Z_2\times\Z_2\), the matrices \(A_1\) and \(A_2\) in the proof of \Cref{thm_ar} below lie in \(\USGA(X)\) and generate \(\operatorname{GL}_2(\Z_2)=\Aut_c(X)\cong S_3\) (see \Cref{rmk}). Hence neither \(\USGA(X)\) nor \(\Aut_c(X)\) provides a weak representing object for the actions on \(X\) in the variety \(\Nil_2(\Grp)\). In this section, we show that it is nevertheless possible to construct such a weak representing object starting from the images of the morphisms
\[
    U(B) \to \Aut_c(X)
\]
which correspond to the actions of \(B\) on \(X\) in \(\Nil_2(\Grp)\). Our aim is therefore to prove that \(\Nil_2(\Grp)\) is a weakly action representable category. We start by observing the following.

\begin{proposition}\label{thm_ar}
    The category \(\Nil_2(\Grp)\) is action accessible and not action representable.
\end{proposition}

\begin{proof}
    The category \(\Nil_2(\Grp)\) is action accessible, since it is an Orzech category of interest (see~\cite{Montoli}). To see that it is not action representable, let \(X=\Z_2\times\Z_2\) and suppose that \(X\) admits an actor~\([X]\). Let
    \[
        \psi=\rho_{[X]}(\id_{[X]})\colon U([X])\to\Aut_c(X)
    \]
    be the acting morphism corresponding to the universal split extension. By naturality of~\(\rho\), any acting morphism \(\varphi\colon U(B)\to\Aut_c(X)\) factors as \(\psi\circ U(\chi)\), where \(\chi\colon {B\to[X]}\) classifies the corresponding action; in particular, \(\Imm(\varphi)\leq\Imm(\psi)\). Now consider the matrices
    \[
        A_1=\begin{pmatrix} 1 & 1\\ 0 & 1\end{pmatrix}
        \qquad\text{and}\qquad
        A_2=\begin{pmatrix} 1 & 0\\ 1 & 1\end{pmatrix}
    \]
    in \(\Aut_c(X)=\operatorname{GL}_2(\Z_2)\). Since \((A_i-I_2)^2=0\), the homomorphism \({\Z_2\to\Aut_c(X)}\) sending the generator to~\(A_i\) satisfies~\eqref{act_mor}, so that \(A_1\), \(A_2\in\Imm(\psi)\). By \Cref{main_thm}, \(\Imm(\psi)\) is abelian, which contradicts \(A_1A_2\neq A_2A_1\).
\end{proof}

We now explain how the problem of finding a weak representation of the actions on a \(2\)-nilpotent group \(X\) is related to the so-called \emph{amalgamation property} (AP)~\cite{kiss}.

We recall that for a diagram in a category \(\C\), an \emph{amalgam} is a monic cocone, i.e., a cocone which is a monomorphic natural transformation. This means that each component of that cocone is a monomorphism, which implies that all the morphisms of the given diagram were monomorphisms to begin with. Note that in a category with colimits, an amalgam for a diagram exists if and only if its colimit cocone is such an amalgam. A category is said to have the amalgamation property~(AP) when every span of monomorphisms admits an amalgam; equivalently, for each pushout square
\[
    \xymatrix{I \ar[r]^t \ar[d]_s & T\ar[d]^-{\iota_T}\\
    S \ar[r]_-{\iota_S} & S+_IT}
\]
if \(s\) and \(t\) are monomorphisms then so are \(\iota_S\) and \(\iota_T\). It is well known that the categories of groups, abelian groups and Lie algebras satisfy the condition (AP), whereas, for instance, the category of associative algebras does not. We refer the reader to~\cite{kiss} for an overview of examples and references.

The relationship between action representability and (AP) was first investigated in~\cite{BJK2}, where it was proved that, for any Orzech category of interest, action representability is equivalent to (AP) for \emph{protosplit} monomorphisms, i.e., monomorphisms arising as kernels of split extensions. Moreover, in~\cite{WRAAlg}, the authors showed that the existence of a weak representation of the actions on a non-associative algebra \(X\) is related to the amalgamation property for a certain diagram of subalgebras of the external weak actor \(\E(X)\). For our purposes we shall need a slightly stronger notion than that of an amalgam. Following~\cite{kiss}, an amalgam \((\mu_S\colon S\to T,\ \mu_{S'}\colon S'\to T)\) of a span of subgroups \(S\hookleftarrow S\cap S'\hookrightarrow S'\) of a given group is called \emph{strong} when the images of \(S\) and \(S'\) in \(T\) intersect in the image of \(S\cap S'\); in other words, when \(\mu_S(s)=\mu_{S'}(s')\) holds if and only if \(s=s'\). We adopt this condition for an arbitrary family \((S_i)_{i\in I}\) of subgroups of a given group, viewed as a diagram via the inclusions \(S_i\subseteq S_j\): a cocone \((\mu_i\colon S_i\to T)_{i\in I}\) of group homomorphisms is a \emph{strong amalgam} of the family when \(\mu_i(s)=\mu_j(s')\) if and only if \(s=s'\), for all \(i\), \(j\in I\), \(s\in S_i\) and \(s'\in S_j\). Equivalently, the \(\mu_i\) are the restrictions of a single injective function \(\bigcup_{i\in I}S_i\to T\). Note that the ``if'' part follows from the cocone condition when the family is closed under binary intersections, but is an additional requirement in general. An amalgam, being a monic cocone, is only injective on each \(S_i\) separately, which is a strictly weaker condition.

Now, by \Cref{main_thm} each action \(\xi\) of \(B\) on \(X\) in \(\Nil_2(\Grp)\) gives rise to a morphism \(\rho_B(\xi)\colon U(B)\to \Aut_c(X)\) in \(\Grp\). We want to show that \(\xi\) is also determined by a morphism of \(2\)-nilpotent groups \(\tau_B(\xi)\colon B\to T\), where \(T=T(X)\) is a weak representing object of \(X\).

Each morphism \(\rho_B(\xi)\colon U(B)\to \Aut_c(X)\) is the composite of a surjective morphism of \(2\)-nilpotent groups \(\rho'_B(\xi)\colon U(B)\twoheadrightarrow M_\xi\), where \(M_\xi=\Imm(\rho_B(\xi))\), and the canonical inclusion \(M_\xi \hookrightarrow \Aut_c(X)\). We thus find a diagram of \(2\)-nilpotent subgroups of \(\Aut_c(X)\) indexed over the \(2\)-nilpotent group actions on \(X\). We may re-index and view \(\{M_\xi\}_\xi\) as a diagram in \(\Grp\) over the full subcategory, on the objects \(M_\xi\), of the thin category of \(2\)-nilpotent subgroups of \(\Aut_c(X)\). We further observe that, by (ii) of \Cref{main_thm}, each \(M_\xi\) is an abelian group. Thus, \(\{M_\xi\}_\xi\) is a family of abelian subgroups of the group of central automorphisms of~\(X\).

\begin{remark}\label{rmk_strong}
    Suppose that the family \(\{M_\xi\}_\xi\) admits a strong amalgam \(\bigl(T,(\mu_{M_\xi})_\xi\bigr)\) in \(\Nil_2(\Grp)\), and set \(\tau_B(\xi)\coloneqq \mu_{M_\xi}\circ \rho'_B(\xi)\). If \(\tau_B(\xi)=\tau_B(\psi)\), then \(\mu_{M_\xi}\) and \(\mu_{M_\psi}\) have the same image in \(T\), because \(\rho'_B(\xi)\) and \(\rho'_B(\psi)\) are surjective; since the cocone is injective on \(\bigcup_\xi M_\xi\), this forces \(M_\xi=M_\psi\) (here only the ``only if'' part of the definition is used), whence \(\rho'_B(\xi)=\rho'_B(\psi)\) and \(\xi=\psi\). It is precisely here that a monic cocone would not suffice: an amalgam is injective on each \(M_\xi\) separately, but two distinct subgroups of the diagram may well have the same image in \(T\).

    Concretely, in order to obtain a weak representation of \(\SplExt(-,X)\) it is enough to exhibit an object \(T\) of \(\Nil_2(\Grp)\) together with an \emph{injective function}
    \[
        j \colon \Aut_c(X) \to T
    \]
    whose restriction to each subgroup \(M_\xi\) is a group homomorphism: the composites \(j \circ \rho_B(\xi)\) are then morphisms of \(\Nil_2(\Grp)\), and they determine \(\xi\). We construct such a pair \((T,j)\) in the rest of this section.
\end{remark}

We now show that such a pair \((T,j)\) exists, and that it can be described explicitly. The key observation is the following elementary property of central automorphisms.

\begin{lemma}\label{lemma_delta}
    Let \(X\) be a \(2\)-nilpotent group. For every \(f \in \Aut_c(X)\), the map
    \[
        \delta_f \colon X \to \bZ(X) \colon x \mapsto f(x)x^{-1}
    \]
    is a group homomorphism, and \([X,X] \leq \Ker(\delta_f)\). Moreover,
    \begin{equation}\label{eq_delta_comp}
        \delta_{f \circ g}(x)=\delta_f\bigl(\delta_g(x)\bigr)\,\delta_f(x)\,\delta_g(x),
    \end{equation}
    for any \(f\), \(g \in \Aut_c(X)\) and \(x \in X\).
\end{lemma}

\begin{proof}
    Since \(f\) is a central automorphism, \(\delta_f(x) \in \bZ(X)\) for every \(x \in X\). Given \(x\), \(y \in X\), we have
    \[
        \delta_f(xy)=f(x)f(y)y^{-1}x^{-1}=f(x)\delta_f(y)x^{-1}=\delta_f(y)f(x)x^{-1}=\delta_f(x)\delta_f(y),
    \]
    where the last two equalities hold because \(\delta_f(y)\) lies in \(\bZ(X)\) and \(\bZ(X)\) is abelian. Hence \(\delta_f\) is a group homomorphism with abelian codomain, so that \([X,X] \leq \Ker(\delta_f)\).

    As for \eqref{eq_delta_comp}, writing \(g(x)=\delta_g(x)x\) we obtain
    \[
        \delta_{f\circ g}(x)=f\bigl(g(x)\bigr)x^{-1}=f\bigl(\delta_g(x)\bigr)f(x)x^{-1}=\delta_f\bigl(\delta_g(x)\bigr)\,\delta_g(x)\,\delta_f(x),
    \]
    and the three factors may be permuted, since they all lie in the abelian group \(\bZ(X)\).
\end{proof}

Since \(\bZ(X)\) is abelian, every group homomorphism \(X \to \bZ(X)\) factors through the abelianisation of \(X\). We may therefore consider the group
\[
    T(X) \coloneqq \Hom_{\Grp}\bigl(X/[X,X],\bZ(X)\bigr),
\]
which is abelian under pointwise multiplication, hence an object of \(\Nil_2(\Grp)\), together with the function
\[
    j \colon \Aut_c(X) \to T(X) \colon f \mapsto \delta_f
\]
provided by \Cref{lemma_delta}.

\begin{proposition}\label{prop_j}
    Let \(X\) be a \(2\)-nilpotent group.
    \begin{itemize}
        \item[(i)] The function \(j\) is injective.
        \item[(ii)] For a subgroup \(M \leq \Aut_c(X)\), the restriction \(j|_M \colon M \to T(X)\) is a group homomorphism if and only if
              \begin{equation}\label{eq_star}
                  g\bigl(f(x)x^{-1}\bigr)=f(x)x^{-1},
              \end{equation}
              for all  \(f,\, g \in M\) and \(x \in X\).
        \item[(iii)] For every action \(\xi\) of a \(2\)-nilpotent group \(B\) on \(X\), the subgroup \(M_\xi\) satisfies \eqref{eq_star}, so that \(j|_{M_\xi}\) is a group homomorphism.
    \end{itemize}
\end{proposition}

\begin{proof}
    (i) If \(\delta_f=\delta_g\), then \(f(x)=\delta_f(x)x=\delta_g(x)x=g(x)\) for every \(x \in X\), and thus \(f=g\).

    (ii) By \eqref{eq_delta_comp}, the equality \(\delta_{f \circ g}=\delta_f\cdot\delta_g\) holds for all \(f\), \(g \in M\) if and only if \(\delta_f\bigl(\delta_g(x)\bigr)=1\) for all \(f\), \(g \in M\) and all \(x \in X\). Since \(\delta_f(z)=1\) means \(f(z)=z\), and since the condition is required for every ordered pair of elements of \(M\), this is exactly \eqref{eq_star}.

    (iii) Writing \(f=\rho_B(\xi)(b)\) and \(g=\rho_B(\xi)(b')\), condition \eqref{eq_star} for \(M_\xi=\Imm(\rho_B(\xi))\) reads
    \[
        b' \ast \bigl((b \ast x)x^{-1}\bigr)=(b \ast x)x^{-1},
    \]
    which is \eqref{act_mor}. The claim thus follows from (ii) of \Cref{main_thm}.
\end{proof}

\begin{theorem}\label{thm_wra}
    The category \(\Nil_2(\Grp)\) is weakly action representable. More precisely, for any \(2\)-nilpotent group \(X\), the abelian group
    \[
        T(X)=\Hom_{\Grp}\bigl(X/[X,X],\bZ(X)\bigr)
    \]
    is a weak representing object of \(X\), a weak representation being given by
    \[
        \tau_B \colon \SplExt(B,X) \to \Hom_{\Nil_2(\Grp)}\bigl(B,T(X)\bigr) \colon \xi \mapsto j \circ \rho_B(\xi).
    \]
\end{theorem}

\begin{proof}
    Let \(B\) be a \(2\)-nilpotent group and let \(\xi\) be an action of \(B\) on \(X\) in \(\Nil_2(\Grp)\). By \Cref{main_thm}, the map \(\rho_B(\xi)\colon U(B) \to \Aut_c(X)\) is a group homomorphism with image \(M_\xi\), and by (iii) of \Cref{prop_j} the restriction \(j|_{M_\xi}\) is a group homomorphism as well. Hence
    \[
        \tau_B(\xi)=j \circ \rho_B(\xi) \colon B \to T(X) \colon b \mapsto \bigl(x \mapsto (b \ast x)x^{-1}\bigr)
    \]
    is a group homomorphism, and it is a morphism of \(\Nil_2(\Grp)\), since \(T(X)\) is abelian.

    The map \(\tau_B\) is injective. Indeed, if \(\tau_B(\xi)=\tau_B(\psi)\), then
    \[
        j\bigl(\rho_B(\xi)(b)\bigr)=j\bigl(\rho_B(\psi)(b)\bigr)
    \]
    for every \(b \in B\), so that \(\rho_B(\xi)=\rho_B(\psi)\) by (i) of \Cref{prop_j}. Since \(\rho\) is a natural monomorphism, we conclude that \(\xi=\psi\).

    Finally, the collection of maps \(\{\tau_B\}_B\) is natural in \(B\). Let \(f \colon B' \to B\) be a morphism in \(\Nil_2(\Grp)\) and let \(\xi \in \SplExt(B,X)\). By the naturality of \(\rho\) established in \Cref{main_thm}, we have \(\rho_{B'}\bigl(f^*(\xi)\bigr)=\rho_B(\xi)\circ f\), and composing with \(j\) yields
    \[
        \tau_{B'}\bigl(f^*(\xi)\bigr)=j \circ \rho_B(\xi) \circ f=\tau_B(\xi)\circ f,
    \]
    which is the commutativity of the square
    \[
        \begin{tikzcd}[column sep=large]
            \SplExt(B,X) \arrow[r, "\tau_B"] \arrow[d, swap, "{f^\ast}"]
            & \Hom_{\Nil_2(\Grp)}(B,T(X)) \arrow[d,"{- \circ f}"]\\
            \SplExt(B',X) \arrow[r, swap, "\tau_{B'}"] & \Hom_{\Nil_2(\Grp)}(B',T(X)).
        \end{tikzcd}
    \]
    Therefore \(\tau \colon \SplExt(-,X) \rightarrowtail \Hom_{\Nil_2(\Grp)}(-,T(X))\) is a natural monomorphism of functors, and the category \(\Nil_2(\Grp)\) is weakly action representable.
\end{proof}

\begin{remark}
    In the terminology of \Cref{rmk_strong}, \Cref{prop_j} says that the family \(\{M_\xi\}_\xi\) of abelian subgroups of \(\Aut_c(X)\) admits a strong amalgam in \(\Ab\), a canonical choice being \(T(X)\) together with the cocone \(\bigl(j|_{M_\xi}\bigr)_\xi\). We stress that \(j\) itself is not a group homomorphism in general: by \eqref{eq_delta_comp} it fails to be one precisely on those pairs of central automorphisms which do not satisfy \eqref{eq_star}. For instance, if \(X=\Z_2 \times \Z_2\), written additively, then \(\Aut_c(X)=\operatorname{GL}_2(\Z_2)\) and \(T(X)\) is the additive group of all \(2 \times 2\) matrices over \(\Z_2\), with \(j(f)=f-\id_X\); here \(j\) is injective but not additive, and \eqref{eq_star} becomes \((f-\id_X)(g-\id_X)=0\) for all \(f\), \(g \in M\). This is what makes \(\Nil_2(\Grp)\) weakly action representable without being action representable.
\end{remark}

The result of \Cref{thm_wra} settles an open question left in~\cite{XabiMancini}, where it is shown that, for \(k \geq 3\), the category \(\Nil_k(\Grp)\) is not weakly action representable, and at the same time establishes a connection with the case of \(2\)-nilpotent Lie algebras (see~\cite[Theorem 2.21]{WRAAlg}, where it is proved that the variety of \(2\)-nilpotent anti-commutative algebras is weakly action representable).

As explained in~\cite[Corollary~5.3 and Corollary 5.4]{WRA}, the existence of the initial weak representation follows from the existence of a weak representation and the fact that \(\Nil_2(\Grp)\) is a total category~\cite{Street}.

\section{A note on local algebraic cartesian closedness}\label{sec_LACC}

We conclude the article with a proof that the variety \(\Nil_2(\Grp)\) is not \emph{locally algebraically cartesian closed} \LACC{}. As explained in the introduction, this condition was introduced by J.~R.~A.~Gray in~\cite{GrayThesis} in connection with the concept of \emph{algebraic exponentiation}, and it is closely related to the study of internal actions in semi-abelian categories.

Recall that a semi-abelian category \(\C\) is said to be \emph{locally algebraically cartesian closed} when the change-of-base functors in the fibration of points are left adjoints. This is equivalent to the condition that for every object \(B\) of \(\C\), the functor
\[
    B\flat(-)\colon \C \to \C,
\]
where for any object \(X\) the object \(B\flat X\) is defined as the kernel of the canonical morphism
\[
    [\id_B,0]\colon B+X\to B,
\]
is a left adjoint. As explained in~\cite{GrayACS}, when \(\C\) is a variety of algebras, this is further equivalent to each of the following conditions:
\begin{enumerate}
    \item for all \(B\), the functor \(B\flat(-)\) preserves all colimits;
    \item for all \(B\), the functor \(B\flat(-)\) preserves binary sums.
\end{enumerate}
Hence, \(\C\) is \LACC{} if and only if, for every pair of objects \(X\) and \(Y\) of \(\C\), the canonical comparison morphism
\begin{equation}\label{comparison}
    [B \flat \iota_X, B \flat \iota_Y] \colon (B\flat X)+(B\flat Y)\to B\flat(X+Y)
\end{equation}
is an isomorphism, where \(\iota_X\colon X\to X+Y\) and \(\iota_Y\colon Y\to X+Y\) denote the coproduct injections.

Examples of semi-abelian categories satisfying the condition \LACC{} include the category of groups, the variety of Lie algebras over a unital commutative ring~\cite{GrayLie}, and the category of cocommutative Hopf algebras over a field of characteristic zero, see~\cite{acc, hopf}. It was then proved in~\cite{GM-VdL3} that, among varieties of non-associative algebras over an infinite field of characteristic different from \(2\), the only non-abelian \LACC{} variety is the category of Lie algebras. Thus, for varieties of non-associative algebras, action representability is equivalent to the condition \LACC{}. On the other hand, for arbitrary semi-abelian categories this is known to fail in general. Indeed, the category of Boolean rings is action representable~\cite{BJK2} but does not satisfy \LACC{}, as explained in~\cite[Proposition 6.4]{GrayACS}. It remains an open problem whether local algebraic cartesian closedness implies action representability, or whether a counterexample exists. This makes us wonder what happens in the case of \(2\)-nilpotent groups. We prove the following result:

\begin{theorem}
    The variety \(\Nil_2(\Grp)\) is not locally algebraically cartesian closed.
\end{theorem}

\begin{proof}
    It is sufficient to find \(2\)-nilpotent groups \(B\), \(X\), and \(Y\) such that \eqref{comparison} is not a monomorphism. Here, when we write \(G=\langle S \mid R \rangle_{\Nil_2}\), we mean that \(G\) is the group generated by \(S\)
    subject to the relations \(R\) in the variety \(\Nil_2(\Grp)\). Equivalently, \(G\) is the quotient of the free \(2\)-nilpotent group on \(S\) by the normal subgroup generated by \(R\). In
    particular, all commutators are tacitly central.

    Let \(B=\langle b\rangle\), \(X=\langle x\rangle\), and \(Y=\langle y\rangle\) be three copies of the infinite cyclic group~\(\Z\). One may check that
    \begin{align*}
        B \flat X     & = \langle x,u \mid [x,u]=1 \rangle_{\Nil_2} \cong \Z \times \Z,        \\
        B \flat Y     & = \langle y,v \mid [y,v]=1 \rangle_{\Nil_2} \cong \Z \times \Z,        \\
        B \flat (X+Y) & =\langle x,y,u,v \mid [x,u]=[y,u]=[x,v]=[y,v]=[u,v]=1 \rangle_{\Nil_2} \\
                      & \cong H_3 \times \Z \times \Z,
    \end{align*}
    where \(u=[b,x]\), \(v=[b,y]\), and \(H_3\) denotes the discrete Heisenberg group. Note that \(u\) and \(v\) are commutators in the \(2\)-nilpotent group \(B+X+Y\), hence they are central in it. This is why they are central in \(B \flat (X+Y)\), whereas the corresponding generators of \((B \flat X)+(B \flat Y)\) are not.

    We observe that the coproduct \((B \flat X) + (B \flat Y)\) is the group
    \[
        \langle x,y,u,v \mid [x,u]=[y,v]=1 \rangle_{\Nil_2},
    \]
    whose commutator subgroup is the abelian group generated by
    \[
        [x,y], \quad [x,v], \quad [y,u], \quad [u,v].
    \]
    We remark that none of these four commutators is trivial.

    Now consider the comparison map
    \[
        \varphi=[B \flat \iota_X,\, B \flat \iota_Y]\colon (B\flat X)+(B\flat Y)\to B\flat(X+Y).
    \]
    It is defined on generators by
    \[
        x \mapsto x,\quad y \mapsto y,\quad u \mapsto u,\quad v \mapsto v,
    \]
    where we use the same symbols for generators in the domain and codomain.

    Since \(u\) and \(v\) are central in \(B\flat (X+Y)\), it follows that
    \[
        \varphi([x,v])=\varphi([y,u])=\varphi([u,v])=1.
    \]
    Hence \(\Ker(\varphi)\) is non-trivial, so \(\varphi\) is not an isomorphism. Therefore, the category \(\Nil_2(\Grp)\) is not \LACC{}.
\end{proof}

This result suggests a possible direction for future research, namely to investigate whether the failure of the condition \LACC{} extends more generally to arbitrary non-abelian subvarieties of the category \(\Grp\).

\section*{Acknowledgements}
The first and second authors are supported by the ``National Group for Algebraic and Geometric Structures and their Applications'' (GNSAGA -- INdAM), and by the SDF Sustainability Decision Framework Research Project -- MISE decree of 31/12/2021 (MIMIT Dipartimento per le politiche per le imprese -- Direzione generale per gli incentivi alle imprese) -- CUP:~B79J23000530005, COR:~14019279, Lead Partner:~TD Group Italia Srl, Partner:~University of Palermo.

The second author is a Postdoctoral Researcher of the Fonds de la Recherche Scientifique--FNRS.

The third author is a Research Director of the Fonds de la Recherche Scientifique--FNRS.

\section*{Use of generative AI tools}
Generative AI tools (Anthropic's Claude Opus 5.5 and Claude Fable 5.1) were used to check the mathematical arguments of a draft written by the authors, including computer-assisted verification of the examples, and to assist in revising it, with the aim of improving its correctness and presentation; they were not used to generate new mathematical results. The authors reviewed and edited all output and take full responsibility for the content of this article.

\end{document}